\documentclass[11pt]{article}

\usepackage[all]{xy}
\usepackage{amsmath}
\usepackage{amsfonts}
\usepackage{amssymb}
\usepackage{amscd}
\usepackage{amsthm}
\usepackage{latexsym}
\usepackage{amsbsy}
\usepackage{epsfig, setspace}
\usepackage{subfigure}
\usepackage{graphicx}

\newtheorem{theorem}{Theorem}[section]

\newtheorem{example}[theorem]{Example}

\newcommand{\rnc}[2]{\renewcommand{#1}{#2}}
\rnc{\theequation}{\thesection.\arabic{equation}}
\rnc{\[}{\begin{equation}}
\rnc{\]}{\end{equation}}

\newcommand{\si}{\psi}
\newcommand{\ot}{\otimes}

\begin{document}

\title{A note on Hopf Cyclic Cohomology in Non-symmetric Monoidal Categories} 

\author{Arash Pourkia\\ {\small Huron University College, Western University,}\\ {\small Department of Mathematics, Western University,}\\ {\small London, Ontario, Canada}} 

\date{July 2014}

\maketitle

\begin{abstract}
In our previous work, ``Hopf cyclic cohomology in braided monoidal categories", we extended the formalism of Connes and Moscovici's Hopf cyclic cohomology and the more general case of Hopf cyclic cohomology with coefficients to the context of abelian braided monoidal categories. In this paper we go one step further in reducing the restriction of the ambient  category $\mathcal{C}$ being symmetric. We let  $\mathcal{C}$ to be non-symmetric but assume only the restriction on the braid map $\psi_{H \otimes H}$ for the Hopf algebra $H$ in $\mathcal{C}$ which is the main player in the theory.  In the case of Hopf cyclic cohomology with coefficients, i.e., in general for the triple $(H,C,M)$, we also need the restriction on $\psi_{H \otimes M}$, $\psi_{H \otimes C}$ and $\psi_{M \otimes C}$. We present a family of examples of non-symmetric categories in which many objects with the property $\psi^2=id$ exist (anyonic vector spaces).
\end{abstract}

\section{Introduction} 
In ~\cite{cm2,cm3,cm4},  Connes and Moscovici, motivated by 
transverse index theory for foliations, defined a cohomology theory
of cyclic type for Hopf algebras endowed with a modular pair in involution (MPI). This theory was later extended in ~\cite{hkrs1,hkrs2} to the more general case of Hopf cyclic cohomology with coefficients, by introducing the notion of stable anti Yetter-Drinfeld (SAYD) modules.

In \cite{pl1hha} we extended all these formalisms of Hopf cyclic cohomology
to the context of abelian braided monoidal categories. When the braiding is symmetric we associated a cocyclic object to a braided Hopf algebra endowed with a braided
modular pair in involution. When the braiding is non-symmetric we obtained a para-cocyclic object instead of a cocyclic object. To obtain a cocyclic object from a para-cocyclic object one has to work in an appropriate subspace.

In the present paper we reduce the restriction of the ambient  category $\mathcal{C}$ being symmetric, i.e., $\psi_{A \otimes B}\psi_{B \otimes A}=id_{B \otimes A}$, for all $A$ and $B$ in $\mathcal{C}$, to a less restrictive condition.  We show that even in the non-symmetric case as long as we only have the relation $\psi^2_{H \otimes H}=\psi_{H \otimes H}\psi_{H \otimes H}=id_{H \otimes H}$ we will still obtain a cocyclic object. In the case of Hopf cyclic cohomology with coefficients in a $H$-module $M$ we also need to have the relation $\psi_{H \otimes M}\psi_{M \otimes H}=id_{M \otimes H}$. In the most general case, for the triple $(H,C,M)$, we also need similar restrictions on $\psi_{H \otimes C}$ and $\psi_{M \otimes C}$. This is in fact because the main players in the theory are $H$, $M$ and $C$ and nowhere in any of the proofs in \cite{pl1hha} we need more than the above relations. Therefore the proofs of the two theorems, Theorems \ref{BHCfortriplesinnonsymm} and \ref{brdd ver of CM thry in nonsymmetric}, presented here are the same as the proofs of their analogous theorems, Theorems 3.6 and 7.1 in \cite{pl1hha}.

In the last section of this paper we provide a family of examples of such a situation where the ambient category $\mathcal{C}$ (category of anyonic vector spaces) is not symmetric but for many objects, $A$ and $B$ in $\mathcal{C}$ the property $\psi_{A \otimes B}\psi_{B \otimes A}=id_{B \otimes A}$ holds. 

We close the present paper by looking at a particular example of a braided Hopf algebra (in a non-symmetric category) for which $\psi^2=id$. We will prove that in this particular case the braided Hopf cyclic cohomology (defied in Theorem \ref{brdd ver of CM thry in nonsymmetric}) coincides with the (usual) Hopf cyclic cohomology in the sense of Connes and Moscovici \cite{cm2,cm3,cm4}. This is expected because this braided Hopf algebra is also a Hopf algebra in the usual sense, i.e., within  the category of vector spaces. \newline

To keep this note short we don't include any preliminaries. We refer the reader to \cite{pl1hha,l,ml} for all the preliminary definitions and for our convention of the short notations often used in this paper.  \newline

\thanks{The author would like to express his sincere appreciation to Masoud Khalkhali for illuminating discussions and encouragements.}

\section{Hopf cyclic cohomology in non-symmetric categories} \label{hopf cyclc for H C M}
In this section we first present an analogous of Theorem 3.6 in \cite{pl1hha}. We assign a cocyclic object to a braided triple ($H$,\,$C$,\,$M$) in an abelian braided monoidal category $\mathcal{C}$ but with the symmetric condition for $\mathcal{C}$ replaced by only restrictions on $\psi_{H \otimes M}$, $\psi_{H \otimes H}$, $\psi_{H \otimes C}$ and $\psi_{M \otimes C}$ (Theorem \ref{BHCfortriplesinnonsymm}). Then we let $C=H$ and $M=I$, the identity object of $\mathcal{C}$, and present an analogous of Theorem 3.6 in \cite{pl1hha}, but again with the symmetric condition for $\mathcal{C}$ replaced, in this case, by only one restriction on $\psi_{H \otimes H}$ (Theorem  \ref{brdd ver of CM thry in nonsymmetric}). The latter case is in fact the braided version of Connes-Moscovici's Hopf cyclic cohomology in a non-symmetric category $\mathcal{C}$ in which only the relation $\psi^2_{H \otimes H}=id_H$ holds. \newline

The proofs of the two theorems presented here are the same as the proofs of their analogous theorems in \cite{pl1hha}.

\begin{theorem} \label{BHCfortriplesinnonsymm}
Let $\mathcal{C}$ be an abelian braided monoidal category  which is {\bf not necessarily symmetric}. Consider the triple $(H,C,M)$, where $H$ is a Hopf algebra, $C$ is a $H$-module coalgebra and $M$ is a SAYD $H$-module, all in $\mathcal{C}$ for which the following relations hold,  
\[\psi_{H \otimes M}\psi_{M \otimes H}=id_{M \otimes H},\qquad \psi_{H \otimes H}\psi_{H \otimes H}=id_{H \otimes H}, \nonumber \]
\[\psi_{H \otimes C}\psi_{C \otimes H}=id_{C \otimes H},\qquad \psi_{M \otimes C}\psi_{C \otimes M}=id_{C \otimes M}. \nonumber \]
Then  $( C^\bullet_H ,\widetilde{\delta_i}, \widetilde{\sigma_i}, \widetilde{\tau_n})$, defined as follows, is a cocyclic object  in $\mathcal{C}$.
For $n\geq 0$, we let
\[ C^n=C^n(C,M) := M \ot C^{n+1}. \nonumber \]
We define faces $\delta_i:C^{n-1} \rightarrow C^n$, degeneracies $\sigma_i
:C^{n+1} \rightarrow C^n$ and cyclic maps $\tau_n :C^n \to C^n$ by:
\begin{eqnarray*}
\delta_i &=& \begin{cases}
 (1_M,1_{C^i}, \Delta_C, 1_{C^{n-i-1}}) \quad \quad \quad  \quad \quad \textrm{$0 \leq i < n$}\\
 (1_M , \si_{C , C^n})(1_M , \phi_C , 1_{C^n})(\si_{H,M} , 1_{C^{n+1}})(\rho _M , \Delta_C , 1_{C^{n-1}}) & \textrm{$i=n$}
  \end{cases}\\
\sigma_i &=& (1_M , 1_{C^{i+1}} , \varepsilon_C , 1_{C^{n-i}}), \quad \quad \quad \quad \quad  \textrm{$0 \leq i \leq n $}\\
 \tau_n &=& (1_M , \si_{C ,C^n})(1_M , \phi_C , 1_{C^n})(\si_{H,M} , 1_{C^{n+1}})(\rho _M ,1_{C^{n+1}})
\end{eqnarray*}

We form the balanced tensor products (Here is when we need $\mathcal{C}$ to be abelian), 
\[ C^n_H=C^n_H(C,M) := M \ot_H C^{n+1}, \quad \quad n \geq 0, \nonumber  \]
 with induced faces, degeneracies and cyclic maps denoted by $\widetilde{\delta_i}$, $\widetilde{\sigma_i}$ and $\widetilde{\tau_n}$.
\end{theorem}

\begin{proof}
As in the proof of Theorem 3.6 in \cite{pl1hha}.
\end{proof}

Now if we put $M=I$ and $C=H$ Theorem \ref{BHCfortriplesinnonsymm} reduces the braided version of Connes-Moscovici's Hopf cyclic theory \cite{cm2,cm3,cm4}, in non-symmetric monoidal categories, as follows. 

\begin{theorem} \label{brdd ver of CM thry in nonsymmetric} Let $\mathcal{C}$ be an abelian braided (not necessarily symmetric) monoidal category. Let $H$  be a braided Hopf algebra in $\mathcal{C}$ for which, $\psi_{H \otimes H}\psi_{H \otimes H}=id_{H \otimes H}$ (for short we show this by $\psi^2_{H \otimes H}=id$). Then if $H$ is endowed with a BMPI $(\delta, \sigma)$ the following data defines a cocyclic object in $\mathcal{C}$:
 \[C^0(H)=I~~~and~~~C^n(H) = H^n, \,  ~~  n \geq 1, \nonumber  \]
\[
\delta_i = \left\{ \begin{array}{ll}
 (\eta, 1,1,...,1) & \textrm{ $i=0$}\\
 (1,1,...,1, \underset{ i-th}{\Delta}, 1,1,...1) & \textrm{$1 \leq i \leq n-1$}\\
 (1,1,...,1,\sigma) & \textrm{ $i=n$}
  \end{array} \right.
 \nonumber\]
\[\sigma_i =(1,1,..., \underset{(i+1)-th}{\varepsilon}, 1,1...,1), \,~~0 \leq i \leq n~~\label{degen 0 2}   \nonumber\]
\[
\tau_n = \left\{ \begin{array}{ll}
 id_I & \textrm{ $n=0$}\\
 (m_n)(\Delta^{n-1} \widetilde{S} , 1_{H^{n-1}} , \sigma) & \textrm{ $n \neq 0$}
  \end{array} \right.
 \nonumber\]
Here by $m_n$ we mean, $m_1=m$, and for $n \geq 2$:
\[m_n =m_{H^n} = (\underbrace{m,m,...,m}_{n~times})\mathcal{F}_n(\si), \nonumber \] where
\[\mathcal{F}_n(\si):=\prod_{j=1}^{n-1}~(1_{H^j}, \underbrace{\si, \si,...,\si}_{n-j~times},1_{H^j}). \nonumber\]
\end{theorem}

\begin{proof}
As in the proof of Theorem 7.1 in \cite{pl1hha}.
\end{proof}

We will denote the braided Hopf cyclic cohomology of a braided Hopf algebra $H$ in Theorem \ref{brdd ver of CM thry in nonsymmetric} by $BHC^*(H)$, as we denote the usual Hopf cyclic cohomology of a Hopf algebra $H$ by $HC^*(H)$\newline

\section{Examples} \label{anyoncat} 

In this section we provide examples of non-symmetric categories (anyonic vector spaces, \cite{maj2,maj}) in which many objects $A$ with the property $\psi^2_{A\otimes A}=id_{A\otimes A}$ exist. In fact we prove that, more generally, there are many objects $A$ and $B$ for which $\psi_{B \otimes A}\psi_{A \otimes B}=id_{A \otimes B}$.

We need to recall that, \cite{bn,maj2,maj}, if $(H,\,R =R_1 \otimes R_2)$ is a quasitriangular Hopf algebra and $\mathcal{C}$ the category of all left $H$-modules, then $\mathcal{C}$ is a braided monoidal abelian category. Here the monoidal  structure is defined by 
\[h \rhd (v \otimes w)=h^{(1)}\rhd v\otimes h^{(2)} \rhd w , \nonumber \] 
and the braiding map $ \psi_{V \otimes W}$ by
\[\psi_{V \otimes W}(v \otimes w):= (R_2 \rhd w \otimes R_1 \rhd  v) , \label{siforHmod}\]  
for any $V$ and $W$ in $\mathcal{C}$, where $\rhd$ denotes the action of
$H$. \newline

Let $H=\mathbb{C} \mathbb{Z}_n$, the  group (Hopf) algebra of the finite cyclic group, $\mathbb{Z}_n$, of order $n$. In addition to the trivial one, $R=1\otimes 1$, there exists a nontrivial quasitriangular structure for $H=\mathbb{C} \mathbb{Z}_n$ defined by \cite{maj}:
\[R=(1/n)\sum_{a,b=0}^{n-1} e^{\frac{(-2 \pi iab)}{n}} g^a \otimes g^b,  \label{Rforczn}\] where $g$ is the generator of $\mathbb{Z}_n$. The category of all left $H$-modules, denoted here by $\mathcal{C}$, is known as the category of anyonic vector spaces. The objects of $\mathcal{C}$ are of the form $V=\bigoplus_{i=0}^{n-1}\, V_i$. They are $\mathbb{Z}_n$-graded representations of $\mathbb{C} \mathbb{Z}_n$ and the action of $\mathbb{Z}_n$ on $V$ is given by, 
\[g\rhd v= e^\frac{2\pi i |v|}{n} v, \label{actofgonv} \] 
where $|v|=k$ is the degree of the homogeneous elements $v$ in $V_k$. The morphisms of $\mathcal{C}$ are linear maps that preserve the grading. Applying formulas \eqref{Rforczn} and \eqref{actofgonv} to \eqref{siforHmod} will give the formula for the braiding map in $\mathcal{C}$ as:
\[\psi_{V \otimes W} (v \otimes w) = e^\frac{2\pi i |v||w|}{n}\,\, w\otimes v, \label{anyonbraid}\]
where $|v|$ and $|w|$ are the degrees of homogeneous elements $v$ and $w$ in objects $V$ and $W$, respectively. 

This category is not symmetric when $n>2$. But we will now prove that for many values of $n$ one can always find objects $A$ and $B$ in $\mathcal{C}$ such that, $\psi_{B \otimes A}\psi_{A \otimes B}=id_{A \otimes B}$. 

A set of examples is as follows. Let $n=2m^2$, for some integer $m\geq 2$, and let $A=\bigoplus_{i=0}^{n-1}\, A_i$ and $B=\bigoplus_{i=0}^{n-1}\, B_i$ be  objects which are focused only in degrees, $km$, for integers $k\geq 0$, i.e., degrees zero, $m$, $2m$, $3m$, $4m$ and so on. By $A$ being focused only in these degrees we mean, $A_i = 0$ when $i\neq km$, for integers $k\geq 0$. Thus for any two homogeneous elements $x$ in $A$ and $y$ in $B$, 
\[\frac{2\pi i |x||y|}{2m^2} = kl\pi i  \nonumber\]
when, $|x|=km$ and $|y|=lm$, for integers $k,l \geq 0$. This implies that, by formula \eqref{anyonbraid}, $\psi_{A \otimes B} (x \otimes y)$ is equal to either $y\otimes x$ or $-y\otimes x$. Therefore  $\psi_{B \otimes A}\psi_{A \otimes B} (x \otimes y) = x \otimes y$, for all homogeneous elements $x$ in $A$ and $y$ in $B$. 

An example of above case is when $n=18$ and $A=\bigoplus_{i=0}^{17}\, A_i$ and $B=\bigoplus_{i=0}^{17}\, B_i$ are objects focused only in degrees $0$, $3$, $6$, $9$, $12$ and $15$. 

For another set of examples let $n=m^2$ and let $A=\bigoplus_{i=0}^{n-1}\, A_i$ and $B=\bigoplus_{i=0}^{17}\, B_i$ be objects which are focused only in degrees, $km$, for integers $k\geq 0$. Thus for any two homogeneous elements $x$ in $A$ and $y$ in $B$, 
\[\frac{2\pi i |x||y|}{m^2} = 2kl\pi i \nonumber \]
when, $|x|=km$ and $|y|=lm$, for integers $k,l \geq 0$. This implies that, by formula \eqref{anyonbraid}, $\psi_{A \otimes B} (x \otimes y)=y\otimes x$, for all homogeneous elements $x$ in $A$ and $y$ in $B$. Therefore  $\psi_{B \otimes A}\psi_{A \otimes B} (x \otimes y) = x \otimes y$, for all homogeneous elements $x$ in $A$ and $y$ in $B$. An example of this case is when $n=9$ and $A=\bigoplus_{i=0}^{8}\, A_i$ and $B=\bigoplus_{i=0}^{17}\, B_i$ are objects focused only in degrees $0$, $3$ and $6$. \newline 

Relative to examples provided above a somewhat trivial case is when objects of interest, $A=\bigoplus_{i=0}^{n-1}\, A_i$, are focused only in degree zero, i.e., $A_i = 0$ when $i\neq 0$. But for us this case is still interesting because we have the following example which is a concrete example of a braided Hopf algebra $\underline{H}$ in a non-symmetric category $\mathcal{C}$ for which $\psi^2_{\underline{H} \otimes\underline{H}}=id$.

For the following example we recall that \cite{bn,maj}, referring to what we have recalled in the beginning of this section, $H$ itself could be turned into a braided Hopf algebra $\underline{H}$ in $\mathcal{C}$ the category of left $H$-modules. This braided Hopf algebra $\underline{H}$ has the following structure. As an algebra $\underline{H} = H$, with $H$-module structure given by conjugation,
\[ a \rhd h = a^{(1)}h S(a^{(2)}). \label{actforH} \nonumber \]
The rest of Hopf algebra structure on $\underline{H}$ is given by:
\[ \underline{\Delta} (h)= h^{(\underline{1})} \ot h^{(\underline{2})}= 
h^{(1)}S(R_2) \ot R_1 \rhd h^{(2)}, \label{delforH}\]
with counit $\underline{\varepsilon} = \varepsilon$, and antipode
\[\underline{S} (h)= R_2S(R_1\rhd h). \label{SforH} \]

\begin{example} \label{czn}
Let $H=\mathbb{C} \mathbb{Z}_n$ with the nontrivial quasitriangular structure defined by the formula \eqref{Rforczn}. Notice that, since $\Delta(g)=g\otimes g$ and $S(g)=g^{-1}$, the conjugation action of $g$ on all elements of $\underline{H}=\mathbb{C} \mathbb{Z}_n$ is trivial, i.e., $g^a \rhd g^m = g^ag^mg^{-a}=g^m$, for all integers $a$ and $m$. Then using formulas \eqref{siforHmod} and \eqref{Rforczn} we have,  

\begin{eqnarray*}
\psi_{\underline{H} \otimes\underline{H}} (v \otimes w) &=&  
(1/n)\sum_{a,b=0}^{n-1} e^{(-2 \pi iab)/n} g^b \rhd w \otimes g^a \rhd v  \\ 
&=& \big((1/n)\sum_{a,b=0}^{n-1} e^{(-2 \pi iab)/n} \big) w \otimes v =w\otimes v,
\end{eqnarray*}
for all $v$ and $w$ in $\underline{H}$. Here we have used the fact that $\big((1/n)\sum_{a,b=0}^{n-1} e^{(-2 \pi iab)/n} \big)=1$. Therefore $\psi_{\underline{H} \otimes\underline{H}}$ is the usual flip and $\psi^2_{\underline{H} \otimes\underline{H}}=id$.
Also this, in light of formula \eqref{actofgonv}, means that $\underline{H}$ as a $\mathbb{Z}_n$-graded Hopf algebra in $\mathcal{C}$,  is focused only in degree zero. 
\end{example}

Now we prove that in fact in this particular case, $(\underline{H}, \underline{\Delta}, \underline{S})= (H, \Delta, S)$. By formulas \eqref{Rforczn} and \eqref{delforH} we have:

\begin{eqnarray*}
\underline{\Delta} (h) &=&(1/n)\sum_{a,b=0}^{n-1} e^{(-2 \pi iab)/n}  h^{(1)}g^{-b} \otimes h^{(2)}\\
&=& \sum_{b=0}^{n-1}\big( (1/n)\sum_{a=0}^{n-1} e^{(-2 \pi iab)/n}\big) h^{(1)}g^{-b} \otimes h^{(2)}= h^{(1)} \otimes h^{(2)}= \Delta(h), 
\end{eqnarray*}
for all $h$ in $\underline{H}$. Here we have used the fact that, \[\big((1/n)\sum_{a=0}^{n-1} e^{(-2 \pi iab)/n} \big)= \delta_{b,0},\label{fact}\]
which is equal to $1$ if $b=0$ and $0$ otherwise \cite{maj}. Also by formulas \eqref{Rforczn}, \eqref{SforH} and  \eqref{fact} we have, for all $h$ in $\underline{H}$:

\begin{eqnarray*}
\underline{S} (h) &=& (1/n)\sum_{a,b=0}^{n-1} e^{(-2 \pi iab)/n} g^bS(h)= \sum_{b=0}^{n-1}\big( (1/n)\sum_{a=0}^{n-1} e^{(-2 \pi iab)/n}\big) g^b S(h)\\
&=& S(h),
\end{eqnarray*}

What we have shown implies that $BHC^*(\underline{H})=HC^*(H)$, i.e., the braided Hopf cyclic cohomology (defied in Theorem \ref{brdd ver of CM thry in nonsymmetric}) coincides with the Hopf cyclic cohomology for $\underline{H}=H=\mathbb{C} \mathbb{Z}_n$. This was expected as $\mathbb{C} \mathbb{Z}_n$ is also a Hopf algebra in the usual sense, i.e., within  the category of vector spaces. The cohomology $HC^*(\mathbb{C} \mathbb{Z}_n)$ is well known to be $\mathbb{C}$ in even degrees and zero in odd degrees. \newline 

We refer to \cite{pl5} for further examples of braided Hopf algebras and computing their braided Hopf cyclic cohomology (BHC) in symmetric and non-symmetric categories. 


\end{document}